\documentclass[12pt]{amsart}

\usepackage{amstext}
\usepackage{amsthm}
\usepackage{amsmath}
\usepackage{amssymb}
\usepackage{latexsym}
\usepackage{amsfonts}
\usepackage{graphicx}
\usepackage{texdraw}
\usepackage{graphpap}

\usepackage[pagebackref,hypertexnames=false, colorlinks, citecolor=red, linkcolor=red]{hyperref} 
\usepackage[backrefs]{amsrefs}

\input txdtools

\begin{document}


\title[Thin Sequences, Model Spaces, and Douglas Algebras]{Thin Sequences and Their Role in Model Spaces and Douglas Algebras}


\author[P. Gorkin]{Pamela Gorkin$^\dagger$}
\address{Pamela Gorkin, Department of Mathematics\\ Bucknell University\\  Lewisburg, PA  USA 17837}
\email{pgorkin@bucknell.edu}
\thanks{$\dagger$ Research supported in part by Simons Foundation Grant 243653}

\author[B.D.Wick]{Brett D. Wick$^\ddagger$}
\address{Brett D. Wick, School of Mathematics\\ Georgia Institute of Technology\\ 686 Cherry Street\\ Atlanta, GA USA 30332-0160}
\email{wick@math.gatech.edu}
\thanks{$\ddagger$ Research supported in part by a National Science Foundation DMS grant \# 0955432.}


\subjclass[2010]{Primary: 46E22.  Secondary:  30D55, 47A, 46B15.}


\keywords{reproducing kernel, thin sequences, interpolation, asymptotic orthonormal sequence}

%
%

\newcommand{\ci}[1]{_{ {}_{\scriptstyle #1}}}

\newcommand{\norm}[1]{\ensuremath{\left\|#1\right\|}}
\newcommand{\abs}[1]{\ensuremath{\left\vert#1\right\vert}}
\newcommand{\p}{\ensuremath{\partial}}
\newcommand{\pr}{\mathcal{P}}

\newcommand{\pbar}{\ensuremath{\bar{\partial}}}
\newcommand{\db}{\overline\partial}
\newcommand{\D}{\mathbb{D}}
\newcommand{\B}{\mathbb{B}}
\newcommand{\Sp}{\mathbb{S}}
\newcommand{\T}{\mathbb{T}}
\newcommand{\R}{\mathbb{R}}
\newcommand{\Z}{\mathbb{Z}}
\newcommand{\C}{\mathbb{C}}
\newcommand{\N}{\mathbb{N}}
\newcommand{\Hc}{\mathcal{H}}
\newcommand{\scrL}{\mathcal{L}}
\newcommand{\td}{\widetilde\Delta}

\newcommand{\CC}{\mathcal{C}}
\newcommand{\EE}{\mathcal{E}}
\newcommand{\RR}{\mathcal{R}}
\newcommand{\TT}{\mathcal{T}}

\newcommand{\La}{\langle }
\newcommand{\Ra}{\rangle }
\newcommand{\rk}{\operatorname{rk}}
\newcommand{\card}{\operatorname{card}}
\newcommand{\ran}{\operatorname{Ran}}
\newcommand{\osc}{\operatorname{OSC}}
\newcommand{\im}{\operatorname{Im}}
\newcommand{\re}{\operatorname{Re}}
\newcommand{\tr}{\operatorname{tr}}
\newcommand{\vf}{\varphi}

\renewcommand{\qedsymbol}{$\Box$}
\newtheorem{thm}{Theorem}[section]
\newtheorem{lm}[thm]{Lemma}
\newtheorem{cor}[thm]{Corollary}
\newtheorem{conj}[thm]{Conjecture}
\newtheorem{prob}[thm]{Problem}
\newtheorem{prop}[thm]{Proposition}
\newtheorem*{prop*}{Proposition}
\newtheorem{defin}[thm]{Definition}
\theoremstyle{remark}
\newtheorem{rem}[thm]{Remark}
\newtheorem*{rem*}{Remark}

\numberwithin{equation}{section}


%
%

\begin{abstract}
We study thin interpolating sequences $\{\lambda_n\}$ and their relationship to interpolation in the Hardy space $H^2$ and the model spaces $K_\Theta = H^2 \ominus \Theta H^2$, where $\Theta$ is an inner function. Our results, phrased in terms of the functions that do the interpolation as well as Carleson measures, show that under the assumption that $\Theta(\lambda_n) \to 0$ the interpolation properties in $H^2$ are essentially the same as those in $K_\Theta$. 
\end{abstract}

%
%

\maketitle

\section{Introduction and Motivation}

A sequence $\{\lambda_j\}_{j=1}^\infty$ is an \textnormal{interpolating sequence} for $H^\infty$, the space of bounded analytic functions, if for every $w\in\ell^\infty$ there is a function $f\in H^\infty$ such that $$f(z_j) = w_j, ~\mbox{for all}~ j\in\N.$$ Carleson's interpolation theorem says that $\{\lambda_j\}_{j=1}^\infty$ is an interpolating sequence for $H^\infty$ if and only if 
\begin{equation}
\label{Interp_Cond}
\delta = \inf_{j}\delta_j:=\inf_j \left\vert B_j(\lambda_j)\right\vert=\inf_{j}\prod_{k \ne j} \left|\frac{\lambda_j - \lambda_k}{1 - \overline{\lambda}_j \lambda_k}\right| > 0,
\end{equation}
where
$$
B_j(z):=\prod_{k\neq j}\frac{-\overline{\lambda_k}}{\abs{\lambda_k}}\frac{z-\lambda_k}{1-\overline{\lambda}_kz}
$$
denotes the Blaschke product vanishing on the set of points $\{\lambda_k:k\neq j\}$.

In this paper, we  consider sequences that (eventually) satisfy a stronger condition than \eqref{Interp_Cond}. A sequence $\{\lambda_j\}\subset\D$ is \textit{thin} if 
$$
\lim_{j\to\infty}\delta_j:=\lim_{j\to\infty}\prod_{k\neq j}\left\vert\frac{\lambda_j-\lambda_k}{1-\overline{\lambda}_k\lambda_j}\right\vert=1.
$$

Thin sequences are of interest not only because functions solving interpolation for thin interpolating sequences have good bounds on the norm, but also because they are interpolating sequences for a very small algebra: the algebra $QA =  VMO \cap H^\infty$, where $VMO$ is the space of functions on the unit circle with vanishing mean oscillation \cite{W}.

Continuing work in \cite{CFT} and \cite{GPW}, we are interested in understanding these sequences in different settings. This will require two definitions that are motivated by the work of Shapiro and Shields, \cite{SS}, in which they gave the appropriate conditions for a sequence to be interpolating for the Hardy space $H^2$. 

Considering more general Hilbert spaces will require the introduction of reproducing kernels:  In a reproducing kernel Hilbert space $\mathcal{H}$ (see \cite[p. 17]{AM}) we let $K_{\lambda_n}$ denote the kernel corresponding to the point $\lambda_n$; that is, for each function in the Hilbert space we have that $f(\lambda_n)=\left\langle f, K_{\lambda_n}\right\rangle_{\mathcal{H}}$.   If we have an $\ell^2$ sequence $a = \{a_n\}$, we define $$\|a\|_{N, \ell^2} = \left(\sum_{j \ge N} |a_j|^2\right)^{1/2}.$$ The concepts of interest are the following.

A sequence $\{\lambda_n\}\subset\Omega \subseteq \mathbb{C}^n$ is said to be {\it an eventual $1$-interpolating sequence for a reproducing kernel Hilbert space $\mathcal{H}$}, denoted $EIS_{\mathcal{H}}$,  if for every $\varepsilon > 0$ there exists $N$ such that for each $\{a_n\} \in \ell^2$ there exists $f_{N, a} \in \mathcal{H}$ with
$$f_{N, a}(\lambda_n) \norm{K_{\lambda_n}}_{\mathcal{H}}^{-1}=f_{N, a}(\lambda_n) K_{\lambda_n}(\lambda_n)^{-\frac{1}{2}} = a_n ~\mbox{for}~ n \ge N ~\mbox{and}~ \|f_{N, a}\|_{\mathcal{H}} \le (1 + \varepsilon) \|a\|_{N, \ell^2}.$$
A sequence $\{\lambda_n\}$ is said to be a {\it  strong asymptotic interpolating sequence for $\mathcal{H}$}, denoted $AIS_{\mathcal{H}}$, if for all $\varepsilon > 0$ there exists $N$ such that for all sequences $\{a_n\} \in \ell^2$ there exists a function $G_{N, a} \in \mathcal{H}$ such that $\|G_{N, a}\|_\mathcal{H} \le \|a\|_{N,\ell^2}$ and 
$$\|\{G_{N, a}(\lambda_n) K_{\lambda_n}(\lambda_n)^{-\frac{1}{2}}  - a_n\}\|_{N, \ell^2} < \varepsilon \|a\|_{N, \ell^2}.$$ 

 Given a (nonconstant) inner function $\Theta$, we are interested in these sequences in model spaces;  we define the model space for $\Theta$ an inner function by $K_\Theta = H^2 \ominus \Theta H^2$. The reproducing kernel in $K_\Theta$ for $\lambda_0 \in \mathbb{D}$ is 
$$
K_{\lambda_0}^\Theta(z) = \frac{1 - \overline{\Theta(\lambda_0)}{\Theta(z)}}{1 - \overline{\lambda_0}z}
$$ 
and the normalized reproducing kernel is 
$$
k_{\lambda_0}^\Theta(z) = \sqrt{\frac{1 - |\lambda_0|^2}{1 - |\Theta(\lambda_0)|^2}} K_{\lambda_0}^\Theta(z).
$$
Finally, note that $$K_{\lambda_0} = K_{\lambda_0}^\Theta + \Theta \overline{\Theta(\lambda_0)}K_{\lambda_0}.$$
We let $P_\Theta$ denote the orthogonal projection of $H^2$ onto $K_\Theta$. 

We consider thin sequences in these settings as well as in Douglas algebras: Letting $L^\infty$ denote the algebra of essentially bounded measurable functions on the unit circle, a Douglas algebra is a closed subalgebra of $L^\infty$ containing $H^\infty$. It is a consequence of work of Chang and Marshall that a Douglas algebra $\mathcal{B}$ is equal to the closed algebra generated by $H^\infty$ and the conjugates of the interpolating Blaschke products invertible in $\mathcal{B}$, \cites{C, M}.

In this paper, we continue work started in \cite{GM} and \cite{GPW} investigating the relationship between thin sequences, $EIS_{\mathcal{H}}$ and $AIS_{\mathcal{H}}$ where $\mathcal{H}$ is a model space or the Hardy space $H^2$.   In Section~\ref{HSV}, we consider the notion of eventually interpolating and asymptotic interpolating sequences in the model space setting. We show that in reproducing kernel Hilbert spaces of analytic functions on domains in $\mathbb{C}^n$, these two are the same. Given results in \cite{GPW}, this is not surprising and the proofs are similar to those in the $H^\infty$ setting. We then turn to our main result of that section. If we have a Blaschke sequence $\{\lambda_n\}$ in $\mathbb{D}$ and assume that our inner function $\Theta$ satisfies $|\Theta(\lambda_n)| \to 0$, then a sequence $\{\lambda_n\}$ is an $EIS_{K_\Theta}$ sequence if and only if it is an $EIS_{H^2}$ sequence (and therefore $AIS_{K_\Theta}$ sequence if and only if it is an $AIS_{H^2}$).  In Section~\ref{CMMS} we rephrase these properties in terms of the Carleson embedding constants on the model spaces.
Finally, in Section~\ref{asip_algebra}, we recall the definition of Douglas algebras and show that appropriate definitions and conditions are quite different in that setting.

\section{Preliminaries}
Recall that a sequence $\{x_n\}$ in $\Hc$ 
 is {\it complete} if $~\mbox{Span}\{x_n: n \ge 1\} = \mathcal{H}$,  and
{\it asymptotically orthonormal} ($AOS$) if there exists $N_0$ such that for all $N \ge N_0$ there are positive constants $c_N$ and $C_N$ such that 
\begin{eqnarray}
\label{thininequality} 
c_N \sum_{n \ge N} |a_n|^2 \le \left\|\sum_{n \ge N} a_n x_n\right\|^2_{\Hc} \le C_N \sum_{n \ge N} |a_n|^2,
\end{eqnarray}
where $c_N \to 1$ and $ C_N \to 1$ as $N \to \infty$. If we can take $N_0 = 1$, the sequence is said to be an $AOB$; this is equivalent to being $AOS$ and a Riesz sequence.  Finally, the Gram matrix corresponding to $\{x_j\}$ is the matrix $G = \left(\langle x_n, x_m \rangle\right)_{n, m \ge 1}$.
\bigskip

It is well known that if $\{\lambda_n\}$ is a Blaschke sequence with simple zeros and corresponding Blaschke product $B$, then $\{k_{\lambda_n}\}$, where $$k_{\lambda_n}(z)=\frac{(1-\left\vert \lambda_n\right\vert^2)^{\frac{1}{2}}}{(1-\overline{\lambda_n}z)},$$ is a complete minimal system in $K_B$ and we also know that $\{\lambda_n\}$ is interpolating if and only if $\{k_{\lambda_n}\}$ is a Riesz basis.  The following beautiful theorem provides the connection to thin sequences.

\begin{thm}[Volberg, \cite{V}*{Theorem 2}] 
\label{Volberg} 
The following are equivalent:
\begin{enumerate}
\item$\{\lambda_n\}$ is a thin interpolating sequence;
\item
The sequence $\{k_{\lambda_n}\}$ is a complete $AOB$ in $K_B$;
\item
There exist a separable Hilbert space $\mathcal{K}$, an orthonormal basis $\{e_n\}$ for $\mathcal{K}$ and $U, K: \mathcal{K} \to K_B$, $U$ unitary, $K$ compact, $U + K$ invertible, such that 
$$(U + K)(e_n) = k_{\lambda_n} \text{ for all } n \in \N.$$
\end{enumerate}
 \end{thm}

 In \cite{F}*{Section 3} and \cite{CFT}*{Proposition 3.2}, the authors note that \cite{V}*{Theorem 3} implies the following.

\begin{prop}\label{propCFT} Let $\{x_n\}$ be a sequence in $\Hc$. The following are equivalent:
\begin{enumerate}
\item $\{x_n\}$ is an AOB;
\item There exist a separable Hilbert space $\mathcal{K}$, an orthonormal basis $\{e_n\}$ for $\mathcal{K}$ and $U, K: \mathcal{K} \to \mathcal{H}$, $U$ unitary, $K$ compact, $U + K$ left invertible, such that 
$$(U + K)(e_n) = x_n;$$
\item The Gram matrix $G$ associated to $\{x_n\}$ defines a bounded invertible operator of the form $I + K$ with $K$ compact.
\end{enumerate} 
\end{prop}

We also have the following, which we will use later in this paper.

\begin{prop}[Proposition 5.1, \cite{CFT}]\label{prop5.1CFT}
If $\{\lambda_n\}$ is a sequence of distinct points in $\mathbb{D}$ and $\{k_{\lambda_n}^\Theta\}$ is an $AOS$, then $\{\lambda_n\}$ is a thin interpolating sequence. \end{prop}

\begin{thm}[Theorem 5.2, \cite{CFT}]\label{theorem5.2CFT}
Suppose $\sup_{n \ge 1} |\Theta(\lambda_n)| < 1$. If $\{\lambda_n\}$ is a thin interpolating sequence, then either

(i) $\{k_{\lambda_n}^\Theta\}_{n\ge1}$ is an $AOB$ or

(ii) there exists $p \ge 2$ such that $\{k_{\lambda_n}^\Theta\}_{n \ge p}$ is a complete $AOB$ in $K_\Theta$.
\end{thm}

\section{Hilbert Space Versions}
\label{HSV}

\subsection{Asymptotic and Eventual Interpolating Sequences} 

\label{asip}

Let $\mathcal{H}$ be a reproducing kernel Hilbert space of analytic functions over a domain $\Omega\subset\C^n$ with reproducing kernel $K_\lambda$ at the point $\lambda \in \Omega$.  We define two properties that a sequence $\{\lambda_n\}\subset \Omega$ can have.

\begin{defin}
A sequence $\{\lambda_n\}\subset\Omega$ is an \textnormal{eventual $1$-interpolating sequence for $\mathcal{H}$}, denoted $EIS_{\mathcal{H}}$,  if for every $\varepsilon > 0$ there exists $N$ such that for each $\{a_n\} \in \ell^2$ there exists $f_{N, a} \in \mathcal{H}$ with
$$f_{N, a}(\lambda_n) \norm{K_{\lambda_n}}_{\mathcal{H}}^{-1}=f_{N, a}(\lambda_n) K_{\lambda_n}(\lambda_n)^{-\frac{1}{2}} = a_n ~\mbox{for}~ n \ge N ~\mbox{and}~ \|f_{N, a}\|_{\mathcal{H}} \le (1 + \varepsilon) \|a\|_{N, \ell^2}.$$
\end{defin} 

\begin{defin} A sequence $\{\lambda_n\}\subset\Omega$ is a \textnormal{strong asymptotic interpolating sequence for $\mathcal{H}$}, denoted $AIS_{\mathcal{H}}$, if for all $\varepsilon > 0$ there exists $N$ such that for all sequences $\{a_n\} \in \ell^2$ there exists a function $G_{N, a} \in \mathcal{H}$ such that $\|G_{N, a}\|_\mathcal{H} \le \|a\|_{N,\ell^2}$ and 
$$\|\{G_{N, a}(\lambda_n) K_{\lambda_n}(\lambda_n)^{-\frac{1}{2}}  - a_n\}\|_{N, \ell^2} < \varepsilon \|a\|_{N, \ell^2}.$$ \end{defin}

We now wish to prove Theorem~\ref{EISiffASI} below. The proof, which is a modification of the proof of the open-mapping theorem, also yields a proof of the following proposition.

\begin{prop}\label{Banachspace}

Let $X$ and $Y$ be Banach spaces and let $T: X \to Y$ be a bounded operator and $\varepsilon > 0$. If
$$\sup_{\|y\| = 1} \inf_{\|x\| \le 1} \|Tx - y\| < \varepsilon < 1,$$ then for all $y \in Y$, 
there exists $x \in X$ such that $\|x\| \le \frac{1}{1 - \varepsilon} \|y\|$ and $Tx = y$.
\end{prop}

Theorem~\ref{EISiffASI} follows from Proposition~\ref{Banachspace}, but doing so requires dealing with several technicalities that obfuscate the underlying ideas, and so we present a direct proof of our desired implication. When we turn to Banach algebras, the corresponding implication (in Theorem~\ref{main_algebra}) will be a direct consequence of Proposition~\ref{Banachspace}. We thank the referee for pointing out Proposition~\ref{Banachspace} to us.

\begin{thm}\label{EISiffASI} 
Let $\mathcal{H}$ be a reproducing kernel space of analytic functions over the domain $\Omega\subset\mathbb{C}^n$ with reproducing kernel at the point $\lambda$ given by $K_\lambda$.  Then $\{\lambda_n\}$ is an $EIS_{\mathcal{H}}$ sequence if and only if $\{\lambda_n\}$ is an $AIS_{\mathcal{H}}$.
\end{thm}

\begin{proof} 
If a sequence is an $EIS_{\mathcal{H}}$, then it is trivially $AIS_{\mathcal{H}}$, for given $\varepsilon > 0$ we may take $G_{N, a} = \frac{f_{N, a}}{(1 + \varepsilon)}$. 

For the other direction, suppose $\{\lambda_n\}$ is an $AIS_{\mathcal{H}}$ sequence. Let $\varepsilon > 0$, $N := N(\varepsilon)$, and  $\{a_j\}:=\{a_{j}^{(0)}\}$ be any sequence. First choose $f_0 \in \mathcal{H}$ so that for $n \ge N$ we have
$$\|\{K_{\lambda_n}(\lambda_n)^{-\frac{1}{2}} f_0(\lambda_n) - a_{n}^{(0)}\}\|_{N, \ell^2} < \frac{\varepsilon}{1+\varepsilon} \|a\|_{N, \ell^2}$$ and $$\|f_0\|_{\mathcal{H}} \le \|a\|_{N,\ell^2}.$$
Now let $a_{n}^{(1)} = a_{n}^{(0)} - K_{\lambda_n}(\lambda_n)^{-\frac{1}{2}} f_0(\lambda_n)$. Note that $\|a^{(1)}\|_{N, \ell^2} < \frac{\varepsilon}{1+\varepsilon} \|a\|_{N, \ell^2}$. Since we have an $AIS_{\mathcal{H}}$ sequence, we may choose $f_1$ such that for $n \ge N$ we have
$$\|\{f_1(\lambda_n)K_{\lambda_n}(\lambda_n)^{-\frac{1}{2}}  - a_{n}^{(1)}\}\|_{N, \ell^2} < \frac{\varepsilon}{1+\varepsilon} \|a^{(1)}\|_{N, \ell^2} <  \left(\frac{\varepsilon}{1+\varepsilon}\right)^2\|a\|_{N, \ell^2},$$
and $$\|f_1\|_{\mathcal{H}} \le \|a^{(1)}\|_{N, \ell^2}<\left(\frac{\varepsilon}{1+\varepsilon}\right)\|a\|_{N,\ell^2}.$$ In general, we let
$$a_{j}^{(k)} = -f_{k - 1}(\lambda_j)K_{\lambda_j}(\lambda_j)^{-\frac{1}{2}} + a_{j}^{(k-1)}$$ so that
$$\|a^{(k)}\|_{N, \ell^2} \le \frac{\varepsilon}{1+\varepsilon} \|a^{(k - 1)}\|_{N, \ell^2} \le \left(\frac{\varepsilon}{1+\varepsilon}\right)^2 \|a^{(k-2)}\|_{N, \ell^2} \le \cdots \le \left(\frac{\varepsilon}{1+\varepsilon}\right)^k \|a\|_{N, \ell^2}$$ and 
$$\|f_k\|_{\mathcal{H}} \le \|a^{(k)}\|_{N, \ell^2}<\left(\frac{\varepsilon}{1+\varepsilon}\right)^k\|a\|_{N,\ell^2}.$$ 
Then consider $f(z) = \sum_{k = 0}^\infty f_k(z)$.
Since $f_k(\lambda_j) = \left(a_{j}^{(k)} - a_{j}^{(k+1)}\right)K_{\lambda_j}(\lambda_j)^{\frac{1}{2}}$ and $a_{j}^{(k)} \to 0$ as $k \to \infty$, we have for each $j \ge N$, $$f(\lambda_j) = a_{j}^{(0)} K_{\lambda_j}(\lambda_j)^{\frac{1}{2}} = a_jK_{\lambda_j}(\lambda_j)^{\frac{1}{2}}.$$ Further $\|f\|_{\mathcal{H}}\le \sum_{k = 0}^\infty \left(\frac{\varepsilon}{1+\varepsilon}\right)^{k} \|a\|_{N, \ell^2} = \frac{1}{1 - \frac{\varepsilon}{1+\varepsilon}} \|a\|_{N, \ell^2}=(1+\varepsilon)\|a\|_{N, \ell^2}$.  This proves that $\{\lambda_n\}$ is an $EIS_{\mathcal{H}}$ sequence.
\end{proof}

\subsection{The Hardy and Model Spaces}

We let $\Theta$ denote a nonconstant inner function and apply Theorem~\ref{EISiffASI} to the reproducing kernel Hilbert space $K_{\Theta}$. We also include statements and results about Carleson measures.  Given a non-negative measure $\mu$ on $\D$, let us denote the (possibly infinite) constant
$$
    \CC(\mu) =  \sup_{f \in H^2, f \neq 0} \frac{\|f\|^2_{L^2(\D, \mu)}}{\|f\|^2_2}
$$
as the Carleson embedding constant of $\mu$ on $H^2$ and
$$
    \RR(\mu) =  \sup_{z\in\mathbb{D}} \frac{\|k_z\|_{L^2(\D, \mu)}}{\|k_z\|_2}=\sup_{z} \|k_z\|_{L^2(\D, \mu)}
$$
as the embedding constant of $\mu$ on $k_z$, the normalized reproducing kernel of $H^2$.  It is well-known that $\CC(\mu)\approx \RR(\mu)$, \cites{MR2417425,nikolski}.

\begin{thm}\label{main} Let $\{\lambda_n\}$ be an interpolating sequence  in $\mathbb{D}$ and let $\Theta$ be an inner function.  Suppose that $\kappa:=\sup_{n} \left\vert \Theta(\lambda_n)\right| < 1$.  The following are equivalent:
\begin{enumerate} 

\item $\{\lambda_n\}$ is an $EIS_{H^2}$ sequence\label{eish2};
\item $\{\lambda_n\}$ is a thin interpolating sequence\label{thin1};
\item \label{aob} Either
\begin{enumerate}
\item $\{k_{\lambda_n}^\Theta\}_{n\ge1}$ is an $AOB$, or
\item there exists $p \ge 2$ such that $\{k_{\lambda_n}^\Theta\}_{n \ge p}$ is a complete $AOB$ in $K_\Theta$;
\end{enumerate}
\item $\{\lambda_n\}$ is an $AIS_{H^2}$ sequence\label{Aish2};
\item The measure $$\mu_N = \sum_{k \ge N} (1 - |\lambda_k|^2)\delta_{\lambda_k}$$ is a Carleson measure for $H^2$ with 
Carleson embedding constant $\CC(\mu_N)$ satisfying $\CC(\mu_N) \to 1$ as $N \to \infty$\label{C1};
\item The measure $$\nu_N = \sum_{k \ge N}\frac{(1 - |\lambda_k|^2)}{\delta_k} \delta_{\lambda_k}$$ is a Carleson measure for $H^2$ with embedding constant $\RR_{\nu_N}$ 
on reproducing kernels satisfying $\RR_{\nu_N} \to 1$\label{C2}.

\bigskip
\noindent
Further, \eqref{eis} and \eqref{ais} are equivalent to each other and imply each of the statements above. If, in addition, $ \Theta(\lambda_n) \to0$, then  \eqref{eish2} - \eqref{ais} are equivalent.
\bigskip

\item $\{\lambda_n\}$ is an $EIS_{K_\Theta}$ sequence\label{eis};
\item $\{\lambda_n\}$ is an $AIS_{K_\Theta}$ sequence\label{ais}.
\end{enumerate}
\end{thm}

\begin{proof}
The equivalence between \eqref{eis} and \eqref{ais} is contained in Theorem \ref{EISiffASI}. Similarly, this applies to \eqref{eish2} and \eqref{Aish2}.  In \cite{GPW}*{Theorem 4.5}, the authors prove that \eqref{thin1}, \eqref{C1} and \eqref{C2} are equivalent.  The equivalence between  \eqref{eish2},  \eqref{thin1}, and \eqref{Aish2} is contained in \cite{GPW}.  That \eqref{thin1} implies \eqref{aob} is Theorem~\ref{theorem5.2CFT}. That \eqref{aob} implies \eqref{thin1} also follows from results in \cite{CFT}, for if a sequence is an $AOB$ for some $p \ge 2$ it is an $AOS$ for $p \ge 2$ and hence thin by Proposition~\ref{prop5.1CFT} for $p \ge 2$. This is, of course, the same as being thin interpolating.   Thus, we have the equivalence of equations \eqref{eish2}, \eqref{thin1}, \eqref{aob}, \eqref{Aish2}, \eqref{C1}, and \eqref{C2}, as well as the equivalence of \eqref{eis} and \eqref{ais}. \\

Now we show that \eqref{eis} and \eqref{eish2} are equivalent under the hypothesis that $\Theta(\lambda_n)\to 0$.\\

\eqref{eis}$\Rightarrow$\eqref{eish2}.  Suppose that $\{\lambda_n\}$ is an $EIS_{K_\Theta}$ sequence.  We will prove that this implies it is an $EIS_{H^2}$ sequence, establishing \eqref{eish2}.  \\

Let $\varepsilon>0$ be given.  Choose $\varepsilon^\prime < \varepsilon$ and let $N_1 = N(\varepsilon^\prime)$ be chosen  according to the definition of $\{\lambda_n\}$ being an $EIS_{K_\Theta}$ sequence. Recall that 
$$\kappa_m = \sup_{n \ge m} |\Theta(\lambda_n)| \to 0,$$ so we may assume that we have
chosen $N_1$ so large that 
$$\frac{1 + \varepsilon^\prime}{(1 - \kappa_{N_1}^2)^{1/2}} < 1 + \varepsilon.$$
Define $\{\tilde{a}_n\}$ to be $0$ if $n < N_1$ and $\tilde{a}_n=a_n \left(1-\abs{\Theta(\lambda_n)}^2\right)^{-\frac{1}{2}}$ for $n \ge N_1$. Then $\{\tilde{a}_n\} \in \ell^2$.  
Select $f_a\in K_\Theta\subset H^2$ so that 

$$f_a(\lambda_n) \left(\frac{1-\abs{\Theta(\lambda_n)}^2}{1-\abs{\lambda_n}^2}\right)^{-\frac{1}{2}} = \tilde{a}_n = 
      a_n \left(1-\abs{\Theta(\lambda_n)}^2\right)^{-\frac{1}{2}} \, \textrm{ if } n \ge N_1$$ 
and $$\|f_a\| \le (1 + \varepsilon^\prime) \|\tilde{a}\|_{N_1, \ell^2} \le \frac{(1 + \varepsilon^\prime)}{(1 - \kappa_{N_1}^2)^{1/2}}\|a\|_{N_1, \ell^2} < (1 + \varepsilon) \|a\|_{N_1, \ell^2}.$$
Since $f_a\in K_\Theta$, we have that $f_a\in H^2$, and canceling out the common factor yields that $f_a(\lambda_n)(1-\abs{\lambda_n}^2)^{-\frac{1}{2}}=a_n$ for all $n\geq N_1$.  Thus $\{\lambda_n\}$ is an $EIS_{H^2}$ sequence as claimed.\\

\eqref{eish2}$\Rightarrow$\eqref{eis}. Suppose that $\Theta(\lambda_n) \to 0$ and $\{\lambda_n\}$ is an $EIS_{H^2}$ sequence; equivalently, that $\{\lambda_n\}$ is thin.  We want to show that the sequence $\{\lambda_n\}$ is an $EIS_{K_\Theta}$ sequence.  First we present some observations. \\

First, looking at the definition, we see that we may assume that $\varepsilon > 0$ is small, for any choice of $N$ that works for small $\varepsilon$ also works for larger values.\\

Second, if $f\in H^2$ and we let $\tilde{f}=P_{K_\Theta}f$, then we have that $\norm{\tilde{f}}_2\leq \norm{f}_2$ since $P_{K_\Theta}$ is an orthogonal projection.  Next, we have $P_{K_\Theta} = P_+ - \Theta P_+ \overline{\Theta}$, where $P_+$ is the orthogonal projection of $L^2$ onto $H^2$, so letting $T_{\overline{\Theta}}$ denote the Toeplitz operator with symbol $\overline{\Theta}$ we have
\begin{equation}\label{Toeplitz}
\tilde{f}(z)=f(z)-\Theta(z)T_{\overline{\Theta}}(f)(z).
\end{equation}
In what follows, $\kappa_m := \sup_{n \ge m}|\Theta(\lambda_n)|$ and recall that we assume that $\kappa_m \to 0$. \\

Since $\{\lambda_n\}$ is an $EIS_{H^2}$ sequence, there exists $N_1$ such that for any $a\in\ell^2$ there exists a function $f_0\in H^2$ such that $$f_0(\lambda_n)=a_n\left(\frac{1 - |\Theta(\lambda_n)|^2}{1-\abs{\lambda_n}^2}\right)^\frac{1}{2}~\mbox{for all}~n\geq N_1$$ and $$\norm{f_0}_{2}\leq (1+\varepsilon)\norm{\{a_k (1 - |\Theta(\lambda_k)|^2)^{\frac{1}{2}}\}}_{N_1,\ell^2} \le (1 + \varepsilon)\norm{a }_{N_1,\ell^2}.$$  Here we have applied the $EIS_{H^2}$ property to the sequence $\{a_k(1-\left\vert \Theta(\lambda_k)\right\vert^2)^{\frac{1}{2}}\}\in\ell^2$.  By \eqref{Toeplitz} we have that 
\begin{eqnarray*}
\tilde{f}_0(\lambda_k) & = & f_0(\lambda_k)-\Theta(\lambda_k) T_{\overline{\Theta}}(f_0)(\lambda_k)\\
& = & a_k(1 - |\Theta(\lambda_k)|^2)^\frac{1}{2}(1-\abs{\lambda_k}^2)^{-\frac{1}{2}}-\Theta(\lambda_k) T_{\overline{\Theta}}(f_0)(\lambda_k)\quad\forall k\geq N_1 
\end{eqnarray*}
and $\norm{\tilde{f}_0}_2\leq\norm{f_0}_2\leq (1+\varepsilon)\norm{a}_{N_1,\ell^2}$.  Rearranging the above,  for $k \ge N_1$ we have
\begin{eqnarray*}
\abs{\tilde{f}_0(\lambda_k)(1 - |\Theta(\lambda_k)|^2)^{-\frac{1}{2}}(1-\abs{\lambda_k}^2)^{\frac{1}{2}}-a_k} & = & \abs{\Theta(\lambda_k) T_{\overline{\Theta}}(f_0)(\lambda_k)(1 - |\Theta(\lambda_k)|^2)^{-\frac{1}{2}}(1-\abs{\lambda_k}^2)^{\frac{1}{2}}}\\
& \leq &  \kappa_{N_1}(1 - \kappa_{N_1}^2)^{-\frac{1}{2}} \norm{f_0}_2\\
&\leq& (1+\varepsilon) \kappa_{N_1}(1 - \kappa_{N_1}^2)^{-\frac{1}{2}}  \norm{a}_{N_1,\ell^2}.
\end{eqnarray*}
We claim that $\{a^{(1)}_n\}=\{\tilde{f}_0(\lambda_n)(1 - |\Theta(\lambda_n)|^2)^{-\frac{1}{2}}(1-\abs{\lambda_n}^2)^{\frac{1}{2}} - a_n\}\in\ell^2$ and that there is a constant $N_2$ depending only on $\varepsilon$ and the Carleson measure given by the thin sequence $\{\lambda_n\}$ such that 
\begin{equation}\label{a1} \norm{a^{(1)}}_{N_2,\ell^2}\leq (1+\varepsilon)^2\kappa_{N_1}(1 - \kappa_{N_1}^2)^{-\frac{1}{2}} \norm{a}_{N_1,\ell^2}.\end{equation} 

 Since the sequence $\{\lambda_n\}$ is thin and distinct, it hence generates an $H^2$ Carleson measure with norm at most $(1+\varepsilon)$; that is, we have the existence of $N_2 \ge N_1$ such that $\kappa_{N_2}(1 - \kappa_{N_2}^2)^{-\frac{1}{2}} \le \kappa_{N_1}(1 - \kappa_{N_1}^2)^{-\frac{1}{2}}$ and
\begin{eqnarray*}
\norm{a^{(1)}}_{N_2,\ell^2} & = & \left(\sum_{k\geq N_2} \abs{\Theta(\lambda_k)}^2 \abs{T_{\overline{\Theta}}(f_0)(\lambda_k)}^2 (1 - |\Theta(\lambda_k)|^2)^{-1}(1-\abs{\lambda_k}^2)\right)^{\frac{1}{2}}\\
& \leq &   (1+\varepsilon) \kappa_{N_2}(1 - \kappa_{N_2}^2)^{-\frac{1}{2}} \norm{T_{\overline{\Theta}}f_0}_2 \nonumber\\
& \leq &(1+\varepsilon) \kappa_{N_2}(1 - \kappa_{N_2}^2)^{-\frac{1}{2}}  \norm{f_0}_2\nonumber\\
& \leq & (1+\varepsilon)^2 \kappa_{N_1} (1 - \kappa_{N_1}^2)^{-\frac{1}{2}} \norm{a}_{N_1,\ell^2}\nonumber<\infty,
\end{eqnarray*}
completing the proof of the claim.

We will now iterate these estimates and ideas.  Let $\widetilde{a^{(1)}_n}=-\frac{a^{(1)}_n}{(1 + \varepsilon)^2\kappa_{N_1} (1 - \kappa_{N_1}^2)^{-\frac{1}{2}} }$ for $n \ge N_2$ and $\widetilde{a^{(1)}_n} = 0$ otherwise.  Then from (\ref{a1})
we have that $\norm{\widetilde{a^{(1)}}}_{N_1,\ell^2} = \norm{\widetilde{a^{(1)}}}_{N_2,\ell^2}  \le \norm{a}_{N_1,\ell^2}$.  Since $\{\lambda_n\}$ is an $EIS_{H^2}$ we may choose $f_1\in H^2$  with 
$$f_1(\lambda_n)=\widetilde{a_n^{(1)}}(1 - |\Theta(\lambda_n)|^2)^{\frac{1}{2}}(1-\abs{\lambda_n}^2)^{-\frac{1}{2}}~\mbox{for all}~n\geq N_1$$ and, letting $\widetilde{f}_1 = P_{K_\Theta}(f_1)$, we have $$ \norm{\tilde{f}_1}_{2}\leq\norm{f_1}_{2}\leq (1+\varepsilon)\norm{\widetilde{a^{(1)}}}_{N_1,\ell^2}\leq (1+\varepsilon)\norm{a}_{N_1,\ell^2}.$$
As above,  
\begin{eqnarray*}
\widetilde{f}_1(\lambda_k) & = & f_1(\lambda_k)-\Theta(\lambda_k) T_{\overline{\Theta}}(f_1)(\lambda_k)\\
& = & \widetilde{a_k^{(1)}}(1 - |\Theta(\lambda_k)|^2)^{\frac{1}{2}}(1-\abs{\lambda_k}^2)^{-\frac{1}{2}}-\Theta(\lambda_k) T_{\overline{\Theta}}(f_1)(\lambda_k)\quad\forall k\geq N_1. 
\end{eqnarray*}
And, for $k \ge N_1$ we have
\begin{eqnarray*}
\abs{\tilde{f}_1(\lambda_k)(1 - |\Theta(\lambda_k)|^2)^{-\frac{1}{2}}(1-\abs{\lambda_k}^2)^{\frac{1}{2}}-\widetilde{a_k^{(1)}}} & = & \abs{\Theta(\lambda_k) T_{\overline{\Theta}}(f_1)(\lambda_k)(1 - |\Theta(\lambda_k)|^2)^{-\frac{1}{2}}(1-\abs{\lambda_k}^2)^{\frac{1}{2}}}\\
& \leq &  \kappa_{N_1}(1 - \kappa_{N_1}^2)^{-\frac{1}{2}}  \norm{f_1}_2\\
& \leq & (1+\varepsilon)\kappa_{N_1}(1 - \kappa_{N_1}^2)^{-\frac{1}{2}}   \norm{a}_{N_1,\ell^2}.
\end{eqnarray*}
Using the definition of $\widetilde{a^{(1)}}$, for $k \ge N_2$ one arrives at
\begin{eqnarray*}
\abs{\left((1+\varepsilon)^2\kappa_{N_1}(1 - \kappa_{N_1}^2)^{-\frac{1}{2}} \tilde{f}_1(\lambda_k)+\tilde{f}_0(\lambda_k)\right)(1 - |\Theta(\lambda_k)|^2)^{-\frac{1}{2}}(1-\abs{\lambda_k}^2)^{\frac{1}{2}}-a_k}\\
\leq (1+\varepsilon)^3\kappa_{N_1}^2 (1 - \kappa_{N_1}^2)^{-1}\norm{a}_{N_1,\ell^2}.
\end{eqnarray*}

We continue this procedure, constructing sequences $a^{(j)}\in \ell^2$ and functions $\tilde{f}_j\in K_{\Theta}$ such that 
$$\norm{a^{(j)}}_{N_1,\ell^2}\leq (1+\varepsilon)^{2j}\left(\frac{\kappa_{N_1}}{(1 - \kappa_{N_1}^2)^{\frac{1}{2}}}\right)^j\norm{a}_{N_1,\ell^2},$$
$$
\left\vert \frac{(1-\abs{\lambda_k}^2)^{\frac{1}{2}}}{(1 - |\Theta(\lambda_k)|^2)^{\frac{1}{2}}} \left(\sum_{l=0}^{j} (1+\varepsilon)^{2l}\left(\frac{\kappa_{N_1}}{(1 - \kappa_{N_1}^2)^{\frac{1}{2}}}\right)^{l}\tilde{f}_l(\lambda_k)\right) -a_k \right\vert\leq \left(1+\varepsilon\right)^{2j+1}\left(\frac{\kappa_{N_1}}{(1 - \kappa_{N_1}^2)^{\frac{1}{2}}}\right)^{j+1}\left\Vert a\right\Vert_{N_1,\ell^2},
$$
and $$\norm{\tilde{f}_j}_2\leq (1+\varepsilon)\norm{a}_{N_1,\ell^2}~\mbox{for all}~j\in \N.$$  Define $$F=\sum_{j = 0}^{\infty} (1+\varepsilon)^{2j} \left(\frac{\kappa_{N_1}}{(1 - \kappa_{N_1}^2)^{\frac{1}{2}}}\right)^j \tilde{f}_j.$$  Then $F\in  K_{\Theta}$ since each $\tilde{f}_j\in K_{\Theta}$ and, since $\kappa_m \to 0$, we may assume that $$(1 + \varepsilon)^2\left(\frac{\kappa_{N_1}}{(1 - \kappa_{N_1}^2)^{\frac{1}{2}}}\right) < 1.$$ So,  

$$\norm{F}_2\leq \frac{(1+\varepsilon)}{1 - (1+\varepsilon)^2\left(\frac{\kappa_{N_1}}{\left(1 - \kappa_{N_1}^2\right)^{\frac{1}{2}}}\right)} \norm{a}_{N_1,\ell^2}.$$ 
For this $\varepsilon$, consider $\varepsilon_M < \varepsilon$ with $\frac{(1+\varepsilon_M)}{1 - (1+\varepsilon_M)^2\left(\frac{\kappa_{N_M}}{\left(1 - \kappa_{N_M}^2\right)^{\frac{1}{2}}}\right)}<1+\varepsilon$.  Then, using the process above, we obtain $F_M$ satisfying $F_M \in K_\Theta, \|F_M\|_2 \le  (1 + \varepsilon) \|a\|_{M, \ell^2}$ and $F_M(\lambda_n)\|K_{\lambda_n}\|^{-1}_\mathcal{H} = a_n$ for $n \ge M$.  Taking $N(\varepsilon) = M$, we see that $F_M$ satisfies the exact interpolation conditions, completing the proof of the theorem.

We present an alternate method to prove the equivalence between $(1)$ and $(7)$.  As noted above, by Theorem \ref{EISiffASI} it is true that $(7)\Leftrightarrow (8)$ and thus it suffices to prove that $(1)\Rightarrow (8)\Leftrightarrow (7)$.  Let $\varepsilon>0$ be given.  Select a sequence $\{\delta_N\}$ with $\delta_N\to 0$ as $N\to\infty$.  Since $(1)$ holds, then for large $N$ and $a\in \ell^2$ it is possible to find $f_N\in H^2$ so that 
$$
f_N(a_n)(1-\left\vert\lambda_n\right\vert^2)^{\frac{1}{2}}=\left\Vert a\right\Vert_{N,\ell^2}^{-1} a_n\quad n\geq N
$$
with $\left\Vert f_N\right\Vert_{2}\leq 1+\delta_N$.  Now observe that we can write $f_N=h_N+\Theta g_N$ with $h_N\in K_\Theta$.  Since $h_N$ and $g_N$ are orthogonal projections of $f_N$ onto subspaces of $H^2$, we also have that $\left\Vert h_N\right\Vert_{2}\leq 1+\delta_N$ and similarly for $g_N$.

By the properties of the functions above we have that:
\begin{equation*}
h_N(\lambda_n)\left(\frac{1-\left\vert \lambda_n\right\vert^2}{1-\left\vert \Theta(\lambda_n)\right\vert^2}\right)^{\frac{1}{2}}=f_N(\lambda_n)\left(\frac{1-\left\vert \lambda_n\right\vert^2}{1-\left\vert \Theta(\lambda_n)\right\vert^2}\right)^{\frac{1}{2}}-\Theta(\lambda_n)g_N(\lambda_n)\left(\frac{1-\left\vert \lambda_n\right\vert^2}{1-\left\vert \Theta(\lambda_n)\right\vert^2}\right)^{\frac{1}{2}}.
\end{equation*}
Hence, one deduces that 
\begin{eqnarray*}
\left\Vert \left\{h_N(\lambda_n)\left(\frac{1-\left\vert \lambda_n\right\vert^2}{1-\left\vert \Theta(\lambda_n)\right\vert^2}\right)^{\frac{1}{2}}-\frac{a_n}{\left\Vert a\right\Vert_{N,\ell^2}}\right\}\right\Vert_{N,\ell^2} & \leq & \left\Vert \left\{f_N(\lambda_n)\left(\frac{1-\left\vert \lambda_n\right\vert^2}{1-\left\vert \Theta(\lambda_n)\right\vert^2}\right)^{\frac{1}{2}}-\frac{a_n}{\left\Vert a\right\Vert_{N,\ell^2}}\right\}\right\Vert_{N,\ell^2}\\
&  & + \left\Vert \left\{\Theta(\lambda_n)g_N(\lambda_n)\left(\frac{1-\left\vert \lambda_n\right\vert^2}{1-\left\vert \Theta(\lambda_n)\right\vert^2}\right)^{\frac{1}{2}}\right\}\right\Vert_{N,\ell^2}\\
& \leq & \left\Vert \left\{\frac{a_n}{\left\Vert a\right\Vert_{N,\ell^2}}\left(\left(\frac{1}{1-\left\vert \Theta(\lambda_n)\right\vert^2}\right)^{\frac{1}{2}}-1\right)\right\}\right\Vert_{N,\ell^2}\\
& & +\frac{\sup_{m\geq N}\left\vert\Theta(\lambda_m)\right\vert}{(1-\kappa_N^2)^{\frac{1}{2}}} \left\Vert \left\{g_N(\lambda_n)\left(1-\left\vert \lambda_n\right\vert^2\right)^{\frac{1}{2}}\right\}\right\Vert_{N,\ell^2}.
\end{eqnarray*}
Now for $x$ sufficiently small and positive we have that $\frac{1}{\sqrt{1-x}}-1=\frac{1-\sqrt{1-x}}{\sqrt{1-x}}\lesssim \frac{x}{\sqrt{1-x}}$.  Applying this with $x=\sup_{m\geq N} \left\vert\Theta(\lambda_m)\right\vert$ gives that:
$$
\left\Vert \left\{h_N(\lambda_n)\left(\frac{1-\left\vert \lambda_n\right\vert^2}{1-\left\vert \Theta(\lambda_n)\right\vert^2}\right)^{\frac{1}{2}}-\frac{a_n}{\left\Vert a\right\Vert_{N,\ell^2}}\right\}\right\Vert_{N,\ell^2} \leq \frac{\sup_{m\geq N}\left\vert\Theta(\lambda_m)\right\vert}{(1-\kappa_N^2)^{\frac{1}{2}}} \left(1+\left\Vert \left\{g_N(\lambda_n)\left(1-\left\vert \lambda_n\right\vert^2\right)^{\frac{1}{2}}\right\}\right\Vert_{N,\ell^2}\right).
$$
Define $H_N=(1+\delta_N)^{-1} \left\Vert a\right\Vert_{N,\ell^2} h_N$, and then we have $H_N\in K_{\Theta}$ and $\left\Vert H_N\right\Vert_{2}\leq \left\Vert a\right\Vert_{N,\ell^2}$.  Using the last estimate and adding and subtracting the quantity $\frac{a_n}{(1+\delta_N)}$ yields that:
\begin{eqnarray*}
\left\Vert \left\{H_N(\lambda_n)\left(\frac{1-\left\vert \lambda_n\right\vert^2}{1-\left\vert \Theta(\lambda_n)\right\vert^2}\right)^{\frac{1}{2}}-a_n\right\}\right\Vert_{N,\ell^2}   \leq & &  
\\  \left(\frac{\sup_{m\geq N}\left\vert\Theta(\lambda_m)\right\vert}{(1+\delta_N)(1-\kappa_N^2)^{\frac{1}{2}}} \left(1+\left\Vert \left\{g_N(\lambda_n)\left(1-\left\vert \lambda_n\right\vert^2\right)^{\frac{1}{2}}\right\}\right\Vert_{N,\ell^2}\right)+\delta_N\right)\left\Vert a\right\Vert_{N,\ell^2}.
\end{eqnarray*}
Note that the quantity: 
$$
\left(\frac{\sup_{m\geq N}\left\vert\Theta(\lambda_m)\right\vert}{(1+\delta_N)(1-\kappa_N^2)^{\frac{1}{2}}} \left(1+\left\Vert \left\{g_N(\lambda_n)\left(1-\left\vert \lambda_n\right\vert^2\right)^{\frac{1}{2}}\right\}\right\Vert_{N,\ell^2}\right)+\delta_N\right)\lesssim \delta_N+\sup_{m\geq N}\left\vert\Theta(\lambda_m)\right\vert.
$$
Here we have used that the sequence $\{\lambda_n\}$ is by hypothesis an interpolating sequence and hence: $\left\Vert \left\{g_N(\lambda_n)\left(1-\left\vert \lambda_n\right\vert^2\right)^{\frac{1}{2}}\right\}\right\Vert_{N,\ell^2}\lesssim \left\Vert g_N\right\Vert_{2}\leq 1+\delta_N$.  Since by hypothesis we have that $\delta_N+\sup_{m\geq N}\left\vert\Theta(\lambda_m)\right\vert\to 0$ as $N\to\infty$, it is possible to make this less than the given $\varepsilon>0$, and hence we get a function $H_N$ satisfying the properties for $\{\lambda_n\}$ to be $AIS_{K_\Theta}$.
\end{proof}

\begin{rem} The proof above also gives an estimate on the norm of the interpolating function in the event that $\sup_n |\Theta(\lambda_n)| \le \kappa < 1$, but $(1 + \varepsilon)$ is no longer the best estimate. \end{rem}

\subsection{Carleson Measures in Model Spaces}
\label{CMMS}

From Theorem~\ref{main}, \eqref{C1} and \eqref{C2}, we have a Carleson measure statement for thin sequences in the Hardy space $H^2$. In this section, we obtain an equivalence in model spaces.

We now consider the embedding constants in the case of model spaces. As before, given a positive measure $\mu$ on $\D$, we denote the (possibly infinite) constant

$$
    \CC_{\Theta}(\mu) =  \sup_{f \in K_{\Theta}, f \neq 0} \frac{\|f\|^2_{L^2(\D, \mu)}}{\|f\|^2_2}
$$
as the Carleson embedding constant of $\mu$ on $K_{\Theta}$ and
$$
    \RR_{\Theta}(\mu) =  \sup_{z} \|k^{\Theta}_z\|_{L^2(\D, \mu)}^2
$$
as the embedding constant of $\mu$ on the reproducing kernel of $K_\Theta$ (recall that the kernels $k^{\Theta}_z$ are normalized).  It is known that for general measure $\mu$ the constants $\RR_{\Theta}(\mu)$ and $\CC_{\Theta}(\mu)$ are not equivalent, \cite{NV}.  The complete geometric characterization of the measures for which $\CC_{\Theta}(\mu)$ is finite is contained in \cite{LSUSW}.  However, we always have that
$$
   \RR_\Theta(\mu) \le \CC_\Theta(\mu).
$$
For $N > 1$, let 
$$
\sigma_N = \sum_{k \ge N} \left\Vert K_{\lambda_k}^{\Theta}\right\Vert^{-2}\delta_{\lambda_k}=\sum_{k \ge N} \frac{1-\left\vert \lambda_k\right\vert^2}{1-\left\vert \Theta(\lambda_k)\right\vert^2}\delta_{\lambda_k}.
$$
Note that for each $f \in K_{\Theta}$
\begin{equation}\label{munorm}
        \| f\|^2_{L^2(\D, \sigma_N)} = \sum_{k=N}^\infty \frac{(1 - |\lambda_k|^2)}{(1-\left\vert \Theta(\lambda_k)\right\vert^2)} |f(\lambda_k)|^2 = \sum_{k=N}^\infty |\langle f, k^{\Theta}_{\lambda_k}\rangle|^2,
\end{equation}
and therefore we see that 
\begin{equation}
\label{e:CETests}
1 \le  \RR_\Theta(\sigma_N) \le \CC_\Theta(\sigma_N).
\end{equation}

By working in a restricted setting and imposing a condition on $\{\Theta(\lambda_n)\}$ we have the following.

\begin{thm}
\label{thm:Carleson} Suppose $\Lambda = \{\lambda_n\}$ is a sequence in $\mathbb{D}$ and $\Theta$ is a nonconstant inner function such that $\kappa_m := \sup_{n \ge m}|\Theta(\lambda_n)|\to 0$. 
For $N > 1$, let 
$$
\sigma_N = \sum_{k \ge N} \left\Vert K_{\lambda_k}^{\Theta}\right\Vert^{-2}\delta_{\lambda_k}=\sum_{k \ge N} \frac{1-\left\vert \lambda_k\right\vert^2}{1-\left\vert \Theta(\lambda_k)\right\vert^2}\delta_{\lambda_k}.
$$
Then the following are equivalent:
\begin{enumerate} 
\item $\Lambda$ is a thin sequence;
\item $ \CC_\Theta(\sigma_N) \to 1$ as $N \to \infty$;
\item $ \RR_\Theta(\sigma_N) \to 1$ as $N \to \infty$.
\end{enumerate}
\end{thm}

\begin{proof}
We have $(2)\Rightarrow (3)$ by testing on the function $f=k_z^{\Theta}$ for all $z\in\mathbb{D}$, which is nothing more then \eqref{e:CETests}.

We next focus on $(1)\Rightarrow(2)$.  Let $f \in K_{\Theta}$
and let the sequence $a$ be defined by $a_j = \left\|K_{\lambda_j}\right\|^{-1}f(\lambda_j)$. By \eqref{munorm}, 
$\left\|a\right\|_{N, \ell^2}^2 = \left\| f\right\|^2_{L^2(\D, \sigma_N)}$, and since $\{k_{\lambda_j}^{\Theta}\}$ is an $AOB$, there exists $C_N$ such that
\begin{align*}
\left\|a\right\|_{N,\ell^2}^2 & =  \sum_{j \ge N} \left\|K_{\lambda_j}^\Theta\right\|^{-2}_{K_{\Theta}} |f(\lambda_j)|^2 =  \left\langle f, \sum_{j \ge N} a_j k_{\lambda_j}^\Theta \right\rangle_{K_{\Theta}} \le  \left\| f\right\|_{2} \left\|\sum_{j \ge N} a_j k_{\lambda_j}^\Theta\right\|_{K_{\Theta}} \le  C_N \left\| f\right\|_{2} \left\|a\right\|_{N,\ell^2}.
\end{align*}
By (1) and \cite{CFT}*{Theorem 5.2}, we know that $C_N \to 1$ and since we have established that $\|f\|_{L^2(\D, \sigma_N)} \le C_N \|f\|_2$, (1) $\Rightarrow$ (2) follows.  

An alternate way to prove this is to use Theorem \ref{main}, $(2)\Rightarrow(5)$, and the hypothesis on $\Theta$.  Since it is possible to then show that $\frac{\mathcal{C}_{\Theta}(\sigma_N)}{\mathcal{C}(\mu_N)}\to 1$.  Indeed, given $\varepsilon>0$, we have that $1\leq\mathcal{C}(\mu_M)$ for all $M$, and since $\{\lambda_n\}$ is thin there exists a $N$ such that  $\mathcal{C}(\mu_M)<1+\varepsilon$ for all $M\geq N$.  Hence, $1\leq \mathcal{C}(\mu_M)<1+\varepsilon$ for all $M\geq N$.  These facts easily lead to:
$$
\frac{1}{1+\varepsilon}\leq\frac{\mathcal{C}_{\Theta}(\sigma_M)}{\mathcal{C}(\mu_M)}
$$
Further, since $\Theta$ tends to zero on the sequence $\{\lambda_n\}$ there is an integer, without loss we may take it to be $N$, so that $\frac{1}{1-\left\vert \Theta(\lambda_n)\right\vert^2}<1+\varepsilon$ for all $n\geq N$.  From this we deduce that:
$$
\frac{\mathcal{C}_{\Theta}(\sigma_M)}{\mathcal{C}(\mu_M)}< (1+\varepsilon)\frac{\sup\limits_{f\in K_{\theta}} \sum_{n\geq M} (1-\left\vert \lambda_m\right\vert^2)\left\vert f(\lambda_m)\right\vert^2}{\mathcal{C}(\mu_M)}\leq (1+\varepsilon)
$$
in the last estimate we used that $K_\Theta\subset H^2$ and so the suprema appearing in the numerator is always at most the expression in the denominator.  Combining the estimates we have that for $M\geq N$, that:
$$
\frac{1}{1+\varepsilon}\leq \frac{\mathcal{C}_{\Theta}(\sigma_M)}{\mathcal{C}(\mu_M)}<1+\varepsilon
$$
which yields the conclusion about the ratio tending to $1$ as $N\to \infty$.

Now consider $(3)\Rightarrow (1)$ and compute the quantity $\mathcal{R}_\Theta(\sigma_N)$.  In what follows, we let $\Lambda_N$ denote the tail of sequence, $\Lambda_N=\{\lambda_k: k\geq N\}$.  Note that we have $\left\vert 1-\overline{a}b\right\vert\geq 1-\left\vert a\right\vert$.  Using this estimate we see that:
\begin{eqnarray*}
\sup_{z\in\mathbb{D}} \|k^{\Theta}_z\|_{L^2(\D, \sigma_N)}^2 & = & \sup_{z\in\mathbb{D}} \sum_{k\geq N} \frac{(1-\left\vert \lambda_k\right\vert^2)}{(1-\left\vert \Theta(\lambda_k)\right\vert^2)} \frac{(1-\left\vert z\right\vert^2)}{(1-\left\vert \Theta(z)\right\vert^2)}\frac{\left\vert 1-\Theta(z)\overline{\Theta(\lambda_k)}\right\vert^2}{\left\vert 1-z\overline{\lambda_k}\right\vert^2}\\
 & \geq & \sup_{z\in\mathbb{D}} \sum_{k\geq N} \frac{(1-\left\vert \lambda_k\right\vert^2)(1-\left\vert z\right\vert^2)}{\left\vert 1-z\overline{\lambda_k}\right\vert^2} \frac{(1-\left\vert \Theta(z)\right\vert)(1-\left\vert \Theta(\lambda_k)\right\vert)}{(1-\left\vert \Theta(z)\right\vert^2)(1-\left\vert \Theta(\lambda_k)\right\vert^2)}\\
 & = & \sup_{z\in\mathbb{D}} \sum_{k\geq N} \frac{(1-\left\vert \lambda_k\right\vert^2)(1-\left\vert z\right\vert^2)}{\left\vert 1-z\overline{\lambda_k}\right\vert^2} \frac{1}{(1+\left\vert \Theta(z)\right\vert)(1+\left\vert \Theta(\lambda_k)\right\vert)}\\
  & \geq & \sup_{z\in\Lambda_N} \sum_{k\geq N} \frac{(1-\left\vert \lambda_k\right\vert^2)(1-\left\vert z\right\vert^2)}{\left\vert 1-z\overline{\lambda_k}\right\vert^2} \frac{1}{(1+\left\vert \Theta(z)\right\vert)(1+\left\vert \Theta(\lambda_k)\right\vert)}\\
    & \geq & \frac{1}{(1+\kappa_N)^2}\sup_{z\in\Lambda_N} \sum_{k\geq N} \frac{(1-\left\vert \lambda_k\right\vert^2)(1-\left\vert z\right\vert^2)}{\left\vert 1-z\overline{\lambda_k}\right\vert^2}.
\end{eqnarray*}
By the Weierstrass Inequality, we obtain for $M \ge N$ that
\begin{eqnarray}\label{wi}
 \prod_{k \geq N, k \neq M}  \left|  \frac{\lambda_k - \lambda_M}{1 - \bar \lambda_k \lambda_M}\right|^2 
& = & \prod_{k \geq N, k \neq M} \left(  1-  \frac{(1 - |\lambda_k|^2)(1 - |\lambda_M|^2)}{|1 - \bar \lambda_k \lambda_M|^2} \right)\nonumber\\  
 & \ge & 1 - \sum_{k \geq N, k \neq M}  \frac{(1- |\lambda_M|^2)(1- |\lambda_k|^2)}{ | 1 - \bar \lambda_k \lambda_M|^2}.
\end{eqnarray}
Thus, by \eqref{wi} we have for $M \ge N$,
\begin{eqnarray*}
\frac{1}{(1+\kappa_N)^2}\sup_{z\in\Lambda_N} \sum_{k\geq N} \frac{(1-\left\vert \lambda_k\right\vert^2)(1-\left\vert z\right\vert^2)}{\left\vert 1-z\overline{\lambda_k}\right\vert^2} 
& \ge  & \frac{1}{(1+\kappa_N)^2}\left(\sum_{k \geq N, k\neq M} \frac{(1-\left\vert \lambda_k\right\vert^2)(1-\left\vert \lambda_M\right\vert^2)}{\left\vert 1-\lambda_M\overline{\lambda_k}\right\vert^2} + 1\right)\\
& \ge & \frac{1}{(1+\kappa_N)^2}\left(1 - \prod_{k \geq N, k\neq M}  \left|  \frac{\lambda_k - \lambda_M}{1 - \bar \lambda_k \lambda_M}\right|^2 + 1\right).
\end{eqnarray*}
Now by assumption, recalling that $\kappa_N := \sup_{n \ge N}|\Theta(\lambda_n)|$, we have
$$\lim_{N \to \infty}\sup_{z\in\mathbb{D}} \|k^{\Theta}_z\|_{L^2(\D, \sigma_N)}^2 = 1~\mbox{ and }~\lim_{N \to \infty} \kappa_N = 0,$$ so
$$1 = \lim_{N \to \infty}\sup_{z\in\mathbb{D}} \|k^{\Theta}_z\|_{L^2(\D, \sigma_N)}^2
\ge \lim_{N \to \infty} \frac{1}{(1+\kappa_N)^2}\left(1 - \prod_{k \geq N, k\neq M}  \left|  \frac{\lambda_k - \lambda_M}{1 - \bar \lambda_k \lambda_M}\right|^2 + 1\right) \ge 1.$$
Therefore, for any $M \ge N$
\begin{equation}
\label{e:large}
\prod_{k \geq N, k\neq M}  \left|  \frac{\lambda_k - \lambda_M}{1 - \bar \lambda_k \lambda_M}\right| > 1 - \varepsilon~\mbox{as}~N \to \infty.
\end{equation}
Also, for any $\varepsilon>0$ there is an integer $N_0$ such that for all $M> N_0$ we have:
\begin{equation}
\label{e:bigk}
\prod_{k \geq N_0, k\neq M}  \left|  \frac{\lambda_k - \lambda_M}{1 - \bar \lambda_k \lambda_M}\right| >1-\varepsilon.
\end{equation}
Fix this value of $N_0$, and consider $k<N_0$.  
Further, for $k \ne M$ and $k<N_0$,
\begin{align*}
 1- \rho(\lambda_M, \lambda_k)^2 & =  1- \left|\frac{\lambda_k - \lambda_M}{1 - \bar\lambda_k \lambda_M}\right|^2 = \frac{(1- |\lambda_M|^2)(1- |\lambda_k|^2)}{ | 1 - \bar\lambda_k \lambda_M|^2}\\ 
&=  (1 - |\lambda_k|^2)\frac{(1 - |\lambda_M|^2)}{(1 - |\Theta(\lambda_M)|^2)} \frac{1 - |\Theta(\lambda_M)|^2}{\left\vert 1 - \bar \Theta(\lambda_M) \Theta(\lambda_k)\right\vert^2} \left|\frac{1 - \bar\Theta(\lambda_M) \Theta(\lambda_k)}{1 - \lambda_k \bar\lambda_M}\right|^2\\
& =   \frac{1 - |\Theta(\lambda_M)|^2}{\left\vert 1 - \bar \Theta(\lambda_M) \Theta(\lambda_k)\right\vert^2}(1 - |\lambda_k|^2) |k_{\lambda_M}^\Theta(\lambda_k)|^2\\
& =   \frac{1 - |\Theta(\lambda_M)|^2}{\left\vert 1 - \bar \Theta(\lambda_M) \Theta(\lambda_k)\right\vert^2} (1 - |\Theta(\lambda_k)|^2) \frac{(1 - |\lambda_k|^2)}{1 - |\Theta(\lambda_k)|^2} |k_{\lambda_M}^\Theta(\lambda_k)|^2\\
& \le   \frac{1 - |\Theta(\lambda_M)|^2}{\left\vert 1 - \bar \Theta(\lambda_M) \Theta(\lambda_k)\right\vert^2} \left( \|k_{\lambda_M}^\Theta\|_{L^2(\mathbb{D}, \sigma_M)}^2 - 1\right) \\
& \leq  \frac{1}{(1-\kappa_M)^2}\left( \|k_{\lambda_M}^\Theta\|_{L^2(\mathbb{D}, \sigma_M)}^2 - 1\right)\to 0 ~\mbox{as}~ M \to \infty,
\end{align*}
since $1\leq \|k_{\lambda_N}^\Theta\|_{L^2(\mathbb{D}, \sigma_N)}^2\leq\sup_{z} \|k_{z}^\Theta\|_{L^2(\mathbb{D}, \sigma_N)}^2$ and, by hypothesis, we have that $\kappa_N \to 0$ and 
$\mathcal{R}_{\Theta}(\sigma_N)\to 1$.  Hence, it is possible to choose an integer $M_0$ sufficiently large compared to $N_0$ so that for all $M>M_0$
\begin{equation*}
\rho(\lambda_k,\lambda_{M})>\left(1-\varepsilon\right)^{\frac{1}{N_0}}\quad k<N_0
\end{equation*}
which implies that 
\begin{equation}
\label{e:smallk}
\prod_{k<N_0} \rho(\lambda_k,\lambda_{M})>1-\varepsilon.
\end{equation}

Now given $\varepsilon>0$, first select $N_0$ as above in \eqref{e:bigk}.  Then select $M_0$ so that \eqref{e:large} holds.  Then for any $M>M_0$ by writing the product
$$
\prod_{k\neq M} \rho(\lambda_k,\lambda_M)=\prod_{k<N_0} \rho(\lambda_k,\lambda_{M})\prod_{k>N_0, k\neq M} \rho(\lambda_k,\lambda_M)>(1-\varepsilon)^2.
$$
For the first term in the product we have used \eqref{e:smallk} to conclude that it is greater than $1-\varepsilon$.  And for $M$ sufficiently large, by \eqref{e:bigk}, we have that the second term in the product is greater than $1-\varepsilon$ as well.  Hence, $B$ is thin as claimed.
\end{proof}

\section{Algebra Version} \label{asip_algebra} 

We now compare the model-space version of our results with an algebra version. Theorem~\ref{main} requires that our inner function satisfy $\Theta(\lambda_n) \to 0$ for a thin interpolating sequence $\{\lambda_n\}$ to be an $AIS_{K_\Theta}$ sequence. Letting $B$ denote the Blaschke product corresponding to the sequence $\{\lambda_n\}$, denoting the algebra of continuous functions on the unit circle by $C$, and  letting $H^\infty + C = \{f + g: f \in H^\infty, g \in C\}$ (see \cite{Sarason1} for more on this algebra), we can express this condition in the following way: $\Theta(\lambda_n) \to 0$ if and only if $\overline{B} \Theta \in H^\infty + C$. In other words, if and only if $B$ divides $\Theta$ in $H^\infty + C$, \cites{AG, GIS}. 

We let $\mathcal{B}$ be a Douglas algebra; that is, a uniformly closed subalgebra of $L^\infty$ containing $H^\infty$.   It will be helpful to use the maximal ideal space of our algebra. Throughout $M(\mathcal{B})$ denotes the maximal ideal space of the algebra $\mathcal{B}$; that is, the set of nonzero continuous multiplicative linear functionals on $\mathcal{B}$. 

We now consider thin sequences in uniform algebras. This work is closely connected to the study of such sequences in general uniform algebras (see \cite{GM}) and the special case $B = H^\infty$ is considered in \cite{HIZ}.  With the weak-$\star$ topology, $M(\mathcal{B})$ is a compact Hausdorff space.  In interpreting our results below, it is important to recall that each $x \in M(H^\infty)$ has a unique extension to a linear functional of norm one and, therefore, we may identify $M(\mathcal{B})$ with a subset of $M(H^\infty)$.    In this context, the condition we will require (see Theorem~\ref{main_algebra})  for an $EIS_\mathcal{B}$ sequence to be the same as an $AIS_\mathcal{B}$ sequence is that the sequence be thin near $M(\mathcal{B})$. We take the following as the definition (see \cite{SW}):

\begin{defin} An interpolating sequence $\{\lambda_n\}$ with corresponding Blaschke product $b$ is said to be \textnormal{thin near $M(\mathcal{B})$} if for any $0<\eta < 1$ there is a factorization $b = b_1 b_2$ with $b_1$ invertible in $\mathcal{B}$ and $$|b_2^\prime(\lambda_n)|(1 - |\lambda_n|^2) > \eta$$ for all $n$ such that $b_2(\lambda_n) = 0$. \end{defin}

We will be interested in two related concepts that a sequence can have.  We first introduce a norm on a sequence $\{a_n\}\in \ell^\infty$ that is induced by a second sequence $\{\lambda_n\}$ and a set $\mathcal{O}\supset M(\mathcal{B})$ that is open in $M(H^\infty)$.  Set $I_\mathcal{O}=\{n\in\Z: \lambda_n\in\mathcal{O}\}$.  Then we define
$$
\norm{a}_{\mathcal{O},\ell^\infty}=\sup\{ \abs{a_n}: n\in I_\mathcal{O}\}.
$$

\begin{defin}
A Blaschke sequence $\{\lambda_n\}$ is an \textnormal{eventual $1$-interpolating sequence in a Douglas algebra $\mathcal{B}$}, denoted $EIS_{\mathcal{B}}$,  if for every $\varepsilon > 0$ there exists an open set $\mathcal{O}\supset M(\mathcal{B})$ such that for each $\{a_n\} \in \ell^\infty$ there exists $f_{\mathcal{O}, a} \in H^\infty$ with
$$f_{\mathcal{O}, a}(\lambda_n) = a_n ~\mbox{for}~ \lambda_n\in\mathcal{O}  ~\mbox{and}~ \|f_{\mathcal{O}, a}\|_{\infty} \le (1 + \varepsilon) \|a\|_{\mathcal{O}, \ell^\infty}.$$
\end{defin} 

\begin{defin} A Blaschke sequence $\{\lambda_n\}$ is a \textnormal{strong asymptotic interpolating sequence in a Douglas algebra $\mathcal{B}$}, denoted $AIS_{\mathcal{B}}$, if for all $\varepsilon > 0$ there exists an open set $\mathcal{O}\supset M(\mathcal{B})$ such that for all sequences $\{a_n\} \in \ell^\infty$ there exists a function $G_{\mathcal{O}, a} \in H^\infty$ such that $\|G_{\mathcal{O}, a}\|_{\infty} \le \|a\|_{\mathcal{O},\ell^\infty}$ and 
$$\|\{G_{\mathcal{O}, a}(\lambda_n)  - a_n\}\|_{\mathcal{O}, \ell^\infty} < \varepsilon \|a\|_{\mathcal{O}, \ell^\infty}.$$ \end{defin}

\begin{thm}\label{EISiffASI_algebra} 
Let $\mathcal{B}$ be a Douglas algebra.  Let $\{\lambda_n\}$ be a Blaschke sequence of points in $\D$. Then $\{\lambda_n\}$ is an $EIS_{\mathcal{B}}$ sequence if and only if $\{\lambda_n\}$ is an $AIS_{\mathcal{B}}$.
\end{thm}

\begin{proof} 
If a sequence is an $EIS_{\mathcal{B}}$, then it is trivially $AIS_{\mathcal{B}}$, for given $\varepsilon > 0$ we may take $G_{N, a} = \frac{f_{N, a}}{(1 + \varepsilon)}$. 

For the other direction, suppose $\{\lambda_n\}$ is an $AIS_{\mathcal{B}}$ sequence.
Let $\varepsilon > 0$ be given and let $\varepsilon^\prime < \frac{\varepsilon}{1 + \varepsilon}$. Let $\mathcal{O} \supset M(\mathcal{B})$ denote the open set we obtain from the definition of $AIS_{\mathcal{B}}$ corresponding to $\varepsilon^\prime$. Reordering the points of the sequence in $\mathcal{O}$ so that they begin at $n = 1$ and occur in the same order, we let $T: H^\infty \to \ell^\infty$ be defined by $T(g) = \{g(\lambda_{n})\}$. We let $y_\mathcal{O}$ denote the corresponding reordered sequence. Then $T$ is a bounded linear operator between Banach spaces, so we may use Proposition~\ref{Banachspace} to choose $f \in H^\infty$ so that $Tf = y_\mathcal{O}$ and $\|f\| < \frac{1}{1 - \varepsilon^\prime} \|y_\mathcal{O}\|_{\ell^\infty} < (1 + \varepsilon) \|y\|_{\mathcal{O}, \ell^\infty}$ to complete the proof.
\end{proof}

Letting $\overline{B}$ denote the set of functions with conjugate in $B$, we mention one more set of equivalences. In \cite[Theorem 1]{SW} Sundberg and Wolff showed that an interpolating sequence $\{\lambda_n\}$ is thin  near $M(\mathcal{B})$ if and only if for any bounded sequence of complex numbers $\{w_n\}$ there exists a function in $f \in H^\infty \cap \overline{B}$ such that $f(\lambda_n) = w_n$ for all $n$. 

Finally, we note that Earl (\cite[Theorem 2]{E} or \cite{E2}) proved that given an interpolating sequence for the algebra $H^\infty$ satisfying
$$
\inf_n \prod_{j \ne n} \left|\frac{z_j - z_n}{1 - \overline{z_j} z_n}\right| \ge \delta > 0$$
then for any bounded sequence $\{\omega_n\}$ and 
\begin{equation}\label{Earl}
M > \frac{2 - \delta^2 + 2(1 - \delta^2)^{1/2}}{\delta^2} \sup_n |\omega_n|
\end{equation} there exists a Blaschke product $B$ and a real number $\alpha$ so that
$$M e^{i \alpha} B(\lambda_j) = \omega_j~\mbox{for all}~j.$$

Using the results of Sundberg-Wolff and Earl, we obtain the following theorem.

\begin{thm}
\label{main_algebra} 
Let $\{\lambda_n\}$ in $\mathbb{D}$ be an interpolating Blaschke sequence  and let $\mathcal{B}$ be a Douglas algebra. The following are equivalent:
\begin{enumerate} 
\item $\{\lambda_n\}$ is an $EIS_{\mathcal{B}}$ sequence; \label{EIS_Douglas}
\item $\{\lambda_n\}$ is a $AIS_{\mathcal{B}}$ sequence; \label{AIS_Douglas}
\item $\{\lambda_n\}$ is thin near $M(\mathcal{B})$;\label{nearthin}
\item for any bounded sequence of complex numbers $\{w_n\}$ there exists a function in $f \in H^\infty \cap \overline{B}$ such that $f(\lambda_n) = w_n$ for all $n$.\label{SW}
\end{enumerate}
\end{thm}

\begin{proof}
The equivalence between \eqref{EIS_Douglas} and \eqref{AIS_Douglas} is contained in Theorem \ref{EISiffASI_algebra}. The equivalence of  \eqref{nearthin} and \eqref{SW} is the Sundberg-Wolff theorem.

We next prove that if a sequence is thin near $M(\mathcal{B})$, then it is an $EIS_{\mathcal{B}}$ sequence.  We let $b$ denote the Blaschke product associated to the sequence $\{\lambda_n\}$.  

Given $\varepsilon>0$, choose $\gamma$ so that $$\left(\frac{1 + \sqrt{1 - \gamma^2}}{\gamma}\right)^2 < 1 + \varepsilon.$$  Choose a factorization $b = b_1^\gamma b_2^\gamma$ so that $\overline{b_1^\gamma} \in \mathcal{B}$ and $\delta(b_2) = \inf (1 - |\lambda|^2)|b_2^\gamma \,^\prime(\lambda)| > \gamma$. Since $|b_1^\gamma| = 1$ on $M(\mathcal{B})$ and $\gamma < 1$, there exists an open set $\mathcal{O} \supset M(\mathcal{B})$ such that $|b_1^\gamma| > \gamma$ on $\mathcal{O}$. Note that if $b(\lambda) = 0$ and $\lambda \in \mathcal{O}$, then $b_2(\lambda) = 0$.

The condition on $b_2^\gamma$ coupled with Earl's Theorem (see \eqref{Earl}), gives rise to functions $\{f_k^\gamma\}$ in $H^\infty$ (!), and hence in $\mathcal{B}$ so that
\begin{equation}\label{estimate}
f_j^\gamma(\lambda_k) = \delta_{jk} \, ~\mbox{whenever}~ b_2^\gamma(\lambda_k) = 0 ~\mbox{and}~\sup_{z \in \mathbb{D}}\sum_{j}\abs{f_j^\gamma(z)}\leq \left(\frac{1 + \sqrt{1 - \gamma^2}}{\gamma}\right)^2.
\end{equation}

Now given $a\in\ell^\infty$, choose the corresponding P. Beurling functions (as in \eqref{estimate}) and let
$$
f^\gamma_{\mathcal{O}, a}=\sum_{j} a_j f_j^\gamma.
$$
By construction we have that $f_{\mathcal{O},a}(\lambda_n)=a_n$ for all $\lambda_n\in\mathcal{O}$.  Also, by Earl's estimate \eqref{estimate}, we have that
$$
\norm{f_{\mathcal{O},a}^{\gamma}}_{\infty} \leq (1+\varepsilon)\|a\|_\infty.
$$
Thus, \eqref{nearthin} implies \eqref{EIS_Douglas}.
\color{black}

Finally, we claim \eqref{EIS_Douglas} implies \eqref{nearthin}. Suppose $\{\lambda_n\}$ is a $EIS_{\mathcal{B}}$ sequence. Let $0 < \eta < 1$ be given and choose $\eta_1$ with $1/(1 + \eta_1) > \eta$,  a function $f \in H^\infty$ and $\mathcal{O} \supset M(\mathcal{B})$ open in $M(H^\infty)$ with 
$$f_{\mathcal{O}, n}(\lambda_m) = \delta_{nm}~\mbox{for}~\lambda_m \in \mathcal{O}~\mbox{and}~\|f\|_{\mathcal{O}, n} \le 1 + \eta_1.$$
Let $b_2$ denote the Blaschke product with zeros in $\mathcal{O}$, $b_1$ the Blaschke product with the remaining zeros and let
$$f_{\mathcal{O}, n}(z) = \left(\prod_{j \ne n: b_2(\lambda_j) = 0} \frac{z - \lambda_j}{1 - \overline{\lambda_j}z}\right) h(z),$$ for some $h \in H^\infty.$ Then $\|h\|_{\infty} \le 1 + \eta_1$ and
$$1 = |f_{\mathcal{O}, n}(\lambda_n)| = \left|\left(\prod_{j \ne n; b_2(\lambda_j) = 0} \frac{\lambda_n - \lambda_j}{1 - \overline{\lambda_j}\lambda_n}\right) h(\lambda_n)\right| \le (1 + \eta_1) \prod_{j \ne n} \left|\frac{\lambda_n - \lambda_j}{1 - \overline{\lambda_j}\lambda_n}\right|.$$ Therefore
$$(1 - |\lambda_n|^2)|b_2^\prime(\lambda_n)| = \prod_{j \ne n: b_2(\lambda_j) = 0} \left|\frac{\lambda_n - \lambda_j}{1 - \overline{\lambda_j}\lambda_n}\right| \ge 1/(1 + \eta_1) > \eta.$$ 

Now because we assume that $\{\lambda_n\}$ is interpolating, the Blaschke product $b = b_1 b_2$ with zeros at $\{\lambda_n\}$ will vanish at $x \in M(H^\infty)$ if and only if $x$ lies in the closure of the zeros of $\{\lambda_n\}$, \cite{Hoffman}*{p. 206} or \cite{Garnett}*{p. 379}. Now, if we choose $\mathcal{V}$ open in $M(H^\infty)$ with $M(\mathcal{B}) \subset \mathcal{V} \subset \overline{\mathcal{V}} \subset \mathcal{O}$, then $b_1$ has no zeros in ${\mathcal{V}} \cap \mathbb{D}$ and, therefore, no point of $M(\mathcal{B})$ can lie in the closure of the zeros of $b_1$. So $b_1$ has no zeros on $M(\mathcal{B})$.  Thus we see that $b_1$ is bounded away from zero on $M(\mathcal{B})$ and, consequently, $b_1$ is invertible in $\mathcal{B}$.
\end{proof}

We note that we do not need the full assumption that $b$ is interpolating; it is enough to assume that $b$ does not vanish identically on a Gleason part contained in $M(\mathcal{B})$. Our goal, however, is to illustrate the difference in the Hilbert space and uniform algebra setting and so we have stated the most important setting for our problem. \color{black}




\begin{bibdiv}


\begin{biblist}

\bib{AM}{book}
{
    AUTHOR = {Agler, Jim and McCarthy, John E.},
     TITLE = {Pick interpolation and {H}ilbert function spaces},
    SERIES = {Graduate Studies in Mathematics},
    VOLUME = {44},
 PUBLISHER = {American Mathematical Society, Providence, RI},
      YEAR = {2002},
     PAGES = {xx+308},
      ISBN = {0-8218-2898-3},
   MRCLASS = {47-02 (30D55 30E05 30H05 32A70 46E22 47A20 47A57)},
  MRNUMBER = {1882259 (2003b:47001)},
MRREVIEWER = {D. Sarason},
}
	


\bib{AG}{article}
{
    AUTHOR = {Axler, Sheldon}, 
    Author = {Gorkin, Pamela},
     TITLE = {Divisibility in {D}ouglas algebras},
   JOURNAL = {Michigan Math. J.},
    VOLUME = {31},
      YEAR = {1984},
    NUMBER = {1},
     PAGES = {89--94},
      ISSN = {0026-2285},
}



\bib{CFT}{article}
{
    AUTHOR = {Chalendar, I.},
    Author = {Fricain, E.},
    Author = {Timotin, D.},
     TITLE = {Functional models and asymptotically orthonormal sequences},
   JOURNAL = {Ann. Inst. Fourier (Grenoble)},
    VOLUME = {53},
      YEAR = {2003},
    NUMBER = {5},
     PAGES = {1527--1549}}
     
     \bib{C}{article}
     {
     AUTHOR = {Chang, Sun Yung A.},
     TITLE = {A characterization of {D}ouglas subalgebras},
   JOURNAL = {Acta Math.},
    VOLUME = {137},
      YEAR = {1976},
    NUMBER = {2},
     PAGES = {82--89}}




\bib{E}{article}{
Author = {Earl, J. P},
Title = {On the interpolation of bounded sequences by bounded functions},
Journal = {J. London Math. Soc.},
Volume = {2},
Year = {1970},
Pages = {544--548}
}

\bib{E2}{article}{
 AUTHOR = {Earl, J. P.},
     TITLE = {A note on bounded interpolation in the unit disc},
   JOURNAL = {J. London Math. Soc. (2)},
    VOLUME = {13},
      YEAR = {1976},
    NUMBER = {3},
     PAGES = {419--423}
}

\bib{F}{article}{
  AUTHOR = {Fricain, Emmanuel},
   TITLE = {Bases of reproducing kernels in model spaces},
  JOURNAL = {J. Operator Theory},
 VOLUME = {46},
 YEAR = {2001},
 NUMBER = {3, suppl.},
  PAGES = {517--543}}

\bib{Garnett}{book}
    {author = {Garnett, John B.},
     title = {Bounded analytic functions},
    series= {Pure and Applied Mathematics},
    volume = {96},
publisher = {Academic Press Inc. [Harcourt Brace Jovanovich Publishers]},
   ADDRESS = {New York},
      YEAR = {1981},
     PAGES = {xvi+467},
      ISBN = {0-12-276150-2}}

\bib{GM}{article}
{
    AUTHOR = {Gorkin, Pamela},
    Author={Mortini, Raymond},
     TITLE = {Asymptotic interpolating sequences in uniform algebras},
   JOURNAL = {J. London Math. Soc. (2)},
    VOLUME = {67},
      YEAR = {2003},
    NUMBER = {2},
     PAGES = {481--498}
     }

\bib{GPW}{article}{
AUTHOR = {Gorkin, Pamela},
AUTHOR = {Pott, Sandra},
AUTHOR = {Wick, Brett},
Title = {Thin Sequences and Their Role in $H^p$ Theory, Model Spaces, and Uniform Algebras},
Journal = {Revista Matem\'atica Iberoamericana},
YEAR = {to appear}
}
     

     
\bib{GIS}{article}{
 AUTHOR = {Guillory, Carroll}, Author={Izuchi, Keiji }, Author ={Sarason, Donald},
     TITLE = {Interpolating {B}laschke products and division in {D}ouglas
              algebras},
   JOURNAL = {Proc. Roy. Irish Acad. Sect. A},
    VOLUME = {84},
      YEAR = {1984},
    NUMBER = {1},
     PAGES = {1--7},
}   

\bib{Hoffman}{book}{
    AUTHOR = {Hoffman, Kenneth},
     TITLE = {Banach spaces of analytic functions},
    SERIES = {Prentice-Hall Series in Modern Analysis},
 PUBLISHER = {Prentice-Hall, Inc., Englewood Cliffs, N. J.},
      YEAR = {1962},
     PAGES = {xiii+217}
}

\bib{HIZ}{article}{
    AUTHOR = {Hosokawa, Takuya}
AUTHOR = {Izuchi, Keiji}
AUTHOR =  {Zheng, Dechao},
     TITLE = {Isolated points and essential components of composition
              operators on {$H^\infty$}},
   JOURNAL = {Proc. Amer. Math. Soc.},
  FJOURNAL = {Proceedings of the American Mathematical Society},
    VOLUME = {130},
      YEAR = {2002},
    NUMBER = {6},
     PAGES = {1765--1773},
      ISSN = {0002-9939},
     CODEN = {PAMYAR},
   MRCLASS = {47B33 (30H05)},
  MRNUMBER = {1887024 (2003d:47033)},
MRREVIEWER = {Peter R. Mercer},
       DOI = {10.1090/S0002-9939-01-06233-5},
       URL = {http://dx.doi.org/10.1090/S0002-9939-01-06233-5},
}


\bib{LSUSW}{article}{
  author={Lacey, Michael T.},
  author={Sawyer, Eric T.},
  author={Shen, Chun-Yen},
  author={Uriarte-Tuero, Ignacio},
  author={Wick, Brett D.},
  title={Two Weight Inequalities for the Cauchy Transform from $ \mathbb{R}$ to $ \mathbb{C} _+$},
  eprint={http://arxiv.org/abs/1310.4820},
  journal={J. Inst. Math. Jussieu},
  status={submitted},
  year={2014},
  pages={1--43}
}


\bib{M}{article}{
AUTHOR = {Marshall, Donald E.},
     TITLE = {Subalgebras of {$L^{\infty }$} containing {$H^{\infty }$}},
   JOURNAL = {Acta Math.},
    VOLUME = {137},
      YEAR = {1976},
    NUMBER = {2},
     PAGES = {91--98}
     }
     
\bib{NV}{article}{
   author={Nazarov, F.},
   author={Volberg, A.},
   title={The Bellman function, the two-weight Hilbert transform, and
   embeddings of the model spaces $K_\theta$},
   note={Dedicated to the memory of Thomas H.\ Wolff},
   journal={J. Anal. Math.},
   volume={87},
   date={2002},
   pages={385--414}
}
     
\bib{NOCS}{article}{
   author={Nicolau, Artur},
   author={Ortega-Cerda, Joaquim}, 
   author={Seip, Kristian},
   title={The constant of interpolation},
   journal={Pacific J. Math.},
   volume={213},
   date={2004},
   number={2},
   pages={389--398}
}


\bib{nikolski}{book}{
   author={Nikol{\cprime}ski{\u\i}, N. K.},
   title={Treatise on the shift operator},
   series={Grundlehren der Mathematischen Wissenschaften [Fundamental
   Principles of Mathematical Sciences]},
   volume={273},
   note={Spectral function theory;
   With an appendix by S. V. Hru\v s\v cev [S. V. Khrushch\"ev] and V. V.
   Peller;
   Translated from the Russian by Jaak Peetre},
   publisher={Springer-Verlag},
   place={Berlin},
   date={1986},
   pages={xii+491}
}

\bib{MR2417425}{article}{
   author={Petermichl, Stefanie},
   author={Treil, Sergei},
   author={Wick, Brett D.},
   title={Carleson potentials and the reproducing kernel thesis for
   embedding theorems},
   journal={Illinois J. Math.},
   volume={51},
   date={2007},
   number={4},
   pages={1249--1263}
}



\bib{SS}{article}{
   author={Shapiro, H. S.},
   author={Shields, A. L.},
   title={On some interpolation problems for analytic functions},
   journal={Amer. J. Math.},
   volume={83},
   date={1961},
   pages={513--532}
}

\bib{Sarason1}{article}{
AUTHOR = {Sarason, Donald},
     TITLE = {Algebras of functions on the unit circle},
   JOURNAL = {Bull. Amer. Math. Soc.},
    VOLUME = {79},
      YEAR = {1973},
     PAGES = {286--299},
      }
      

\bib{SW}{article}{
author={Sundberg, C.},
author={Wolff, T.},
title = {Interpolating sequences for $QA_B$},
journal={Trans. Amer. Math. Soc.},
date={1983},
pages={551--581}}

\bib{V}{article}{
AUTHOR = {Vol{\cprime}berg, A. L.},
     TITLE = {Two remarks concerning the theorem of {S}. {A}xler, {S}.-{Y}.
              {A}. {C}hang and {D}. {S}arason},
   JOURNAL = {J. Operator Theory},
    VOLUME = {7},
      YEAR = {1982},
    NUMBER = {2},
     PAGES = {209--218}}



\bib{W}{article}{
   author={Wolff, Thomas H.},
   title={Two algebras of bounded functions},
   journal={Duke Math. J.},
   volume={49},
   date={1982},
   number={2},
   pages={321--328}
}



\end{biblist}
\end{bibdiv}
\end{document}